\documentclass{amsart}
\usepackage{graphicx}
\usepackage{latexsym,amsmath,amssymb}
\newtheorem{thm}{Theorem}[section]

\newtheorem{lem}[thm]{Lemma}

\newtheorem{prop}[thm]{Proposition}
\theoremstyle{definition}
\theoremstyle{remark}

\numberwithin{equation}{section}

\begin{document}

\title[Unbounded composition operators]
{Unbounded composition operators on Orlicz spaces}

\author{\sc\bf M. Namdar Baboli and Y. Estaremi  }
\address{\sc M. Namdar Baboli and Y. Estaremi  }
\email{math.mostafa@yahoo.com}
\email{y.estaremi@gu.ac.ir}

\address{Department of Mathematics, Faculty of Sciences, Golestan University, Gorgan, Iran.}

\thanks{}

\thanks{}



\date{}

\dedicatory{}

\commby{}

\begin{abstract}
In this paper we deal with unbounded composition operators defined in Orlicz spaces. Indeed, we provide some necessary and sufficient condition for densely definedness of composition operators on Orlicz spaces. Also, we will investigate the adjoint of densely defined composition operators and we give some equivalent conditions for it to be densely defined. In addition, we show that densely defined composition operator $C_{\varphi}$ is continuous if and only if it is everywhere defined. Finally, we characterize densely defined continuous composition operators.
\end{abstract}

\maketitle

\section{ \sc\bf Introduction }

Composition operators form a simple but interesting class of operators having interactions with different branches of mathematics and mathematical physics. In mathematics, composition operators commonly occur in the study of shift operators, for example, in the Beurling–Lax theorem and the Wold decomposition. They have been utilised in the dynamical systems to study different types of motions. The ergodic theory and topological dynamics make use of the composition operators in development of their theories.
In this paper we deal with unbounded composition operators defined in Orlicz spaces.\\

The theory of unbounded operators developed in the late 1920s and early 1930s as part of developing a rigorous mathematical framework for quantum mechanics. The theory's development is due to John von Neumann \cite{von} and Marshall Stone \cite{ston}. Von Neumann introduced using graphs to analyze unbounded operators in 1936, \cite{von1}. Subsequently, many mathematicians studied unbounded operators on Banach spaces, especially on function spaces. One can see some recent achievements for unbounded operators in
\cite{bmt, bud, dosi, kadi}. Unbounded composition operators on some function spaces were studied by many authors. For example in \cite{bj} fundamental properties of unbounded composition operators in $L^2$-spaces are studied. Also, in \cite{bjb} and \cite{bud} subnormality of unbounded composition operators are investigated. Moreover, unbounded composition operators on $H^2(B_2)$ are studied in \cite{cw}. In the present paper our goal is to study unbounded composition operators on Orlicz spaces. At first we investigate densely defined composition operators on $L^p$-spaces, that is a special case of Orlicz spaces. In the sequel we provide some necessary and sufficient condition for densely definedness of composition operators on Orlicz spaces. Moreover, we obtain the adjoint operator of composition operator on Orlicz spaces and then we find some conditions under which the adjoint operator is densely defined. Finally, we characterize densely defined continuous composition operators.

\section{ \sc\bf Preliminaries  and basic lemmas }
In this section, we recall the definition of some essential concepts in Orlicz spaces for later use and we refer the interested readers to \cite{kr,rao} for more details.

A no-negative function $\Phi:\mathbb{R}\rightarrow [0,\infty]$ is called a \textit{Young's function}  if it is convex, even ($\Phi(-x)=\Phi(x)$), $\Phi(0)=0$ and $\lim_{x\rightarrow \infty} \Phi(x)=+\infty$. With each \textit{Young's function} $\Phi$, there is another convex function $\Psi:\mathbb{R} \rightarrow [0, \infty]$ having similar properties defined by

$$\Psi(y)=\sup\{x\mid y\mid-\Phi(x): x\geq0\}, \ \ \ \ \ y\in \mathbb{R}.$$

The convex function $\Psi$ is called the \textit{complementary function} to $\Phi$ that it is also a \textit{Young's function}. By definition the pair $(\Phi, \Psi)$ satisfies \textit{Young's inequality}:

$$xy\leq \Phi(x)+\Phi(y), \ \ \ \ \ \ \ \ x,y\in \mathbb{R}.$$

The generalized inverse of $\Phi$ is defined by
$$
\Phi^{-1}(y)=\inf \{ x\geq 0: \Phi(x)> y\} \quad (y\in [0,\infty)).
$$
So by definition we have for all $x\geq0$, $\Phi\big(\Phi^{-1}(x)\big)\leq x$,
and if $\Phi(x)<\infty$, we also have $x\leq\Phi^{-1}\big(\Phi(x)\big)$. The inequalities will turn into equalities when $\Phi$ is a \textit{Young's function} vanishing only at zero and taking only finite values.\\

An especially useful \textit{nice Young's function} $\Phi$, termed an $N$-\textit{function}, is a continuous \textit{Young function} such that $\Phi(x)=0$ if and only if $x=0$ and $\lim_{x\rightarrow 0}\frac{\Phi(x)}{x}=0$, $\lim_{x\rightarrow \infty}\frac{\Phi(x)}{x}=+\infty$, while $\Phi(\mathbb{R})\subset \mathbb{R}^{+}$. Moreover, a function complementary to an $N$-function is again an $N$-function.\\

Let $\Phi$ be a \textit{Young's function}. Then we say $\Phi$ satisfies the
$\Delta_{2}$-condition, if $\Phi(2x)\leq
K\Phi(x) \; ( x\geq x_{0})$  for some constants
$K>0$ and $x_0>0$. Also, it is said to satisfy the
$\Delta'$-condition (respectively, the $\nabla'$-condition), if there exist  $ d>0$
(respectively, $b>0$) and $x_0>0$ such that
$$
\Phi(xy)\leq d\,\Phi(x)\Phi(y) \quad (x,y\geq x_{0})
$$
$$
(\mbox{respectively, }  \Phi(bxy)\geq \Phi(x)\Phi(y) \quad ( x,y\geq x_{0})).
$$
If $x_{0}=0$, these conditions are said to hold
globally. Notice that if $\Phi\in \Delta'$, then  $\Phi\in
\Delta_{2}$.

Let $(X, \Sigma, \mu)$ be a complete $\sigma$-finite measure space and $L^0(\Sigma)$  be the linear space of  equivalence classes of $\Sigma$-measurable real-valued functions on $X$. For $f\in L^0(\Sigma)$ the set $S(f)$ defined by $S(f):=\{x\in X : f(x)\neq
0\}$ is called the support of $f$. For every \textit{Young's function}  $\Phi$, the linear space
$$
L^{\Phi}(\mu)=\left\{f\in L^0(\Sigma):\exists k>0,
\int_X\Phi(kf)d\mu<\infty\right\}
$$
is called an Orlicz space. The functional $N_{\Phi}(.)$ defined by

$$N_{\Phi}(f)=\inf \{k>0:\int_{X}\Phi(\frac{f}{k})d\mu\leq 1\},$$

is a norm on $L^{\Phi}(\mu)$ and is called \textit{guage norm}(or Luxemburge norm). Also, $(L^{\Phi}(\mu), N_{\Phi}(.))$ is a Banach space, the basic measure space $(X,\Sigma,\mu)$ is unrestricted. There is another norm on $L^{\Phi}(\mu)$, defined as follow:

$$\|f\|_{\Phi}=\sup\{\int_{X}\mid fg\mid d\mu: g\in B_{\Psi}\}=\sup\{\mid\int_{X}fg d\mu\mid: g\in B_{\Psi}\},$$
in which $B_{\Psi}=\{g\in L^{\Psi}(\mu): \int_{X}\Psi(\mid g\mid )d\mu\leq 1\}$.
  The norm $\|.\|_{\Phi}$ is called \textit{Orlicz norm} and for any $f\in L^{\Phi}(\mu)$, the inequality
  $$N_{\Phi}(f)\leq \|f\|_{\Phi}\leq2N_{\Phi}(f),$$
holds.
For a \textit{Young's function} $\Phi$, let $\rho_{\Phi}:L^{\Phi}(\mu)\rightarrow \mathbb{R}^{+}$ such that $\rho_{\Phi}(f)=\int_{X}\Phi(f)d\mu$ for all $f\in L^{\Phi}(\mu)$.
Here we recall some facts on convergence of sequences in Orlicz spaces.

\begin{thm}\cite{rao}\label{t23}
Let $\{f_n\}_{n\geq1}$ be a sequence from $L^{\Phi}(\mu)$ and $f\in L^{\Phi}(\mu)$. Then the following assertions hold:\\

(a) If $\|f_n-f\|_{\Phi}\rightarrow 0$ (or equivalently $N_{\Phi}(f_n-f)\rightarrow 0$), then $\rho_{\Phi}(f_n)\rightarrow \rho_{\Phi}(f)$. The converse holds if $\Phi$ is $\triangle_2$-regular.\\

(b) If $\Phi$ is $\triangle_2$-regular Young function, or if $\Phi$ is continuous and concave $\Phi(0)=0$, $\Phi\nearrow$ as well, $\rho_{\Phi}(f_n)\rightarrow \rho_{\Phi}(f)$ as $n\rightarrow \infty$ and $f_n\rightarrow f$ \ \ a.e., or in $\mu$-measure, then $f_n\rightarrow f$ in norm.
\end{thm}

Finally we recall the defintion of conditional expectation. Let $\mathcal{A}\subseteq\Sigma$ be a sub-$\sigma$-finite algebra. The
conditional expectation associated with $\mathcal{A}$ is
the mapping $f\rightarrow E^{\mathcal{A}}f$, defined for all
non-negative, measurable function $f$ as well as for all $f\in
L^1(\Sigma)$ and $f\in L^{\infty}(\Sigma)$, where
$E^{\mathcal{A}}f$, by the Radon-Nikodym theorem, is the unique
$\mathcal{A}$-measurable function satisfying
$$\int_{A}fd\mu=\int_{A}E^{\mathcal{A}}fd\mu, \ \ \ \forall A\in \mathcal{A} .$$
As an operator on $L^{1}({\Sigma})$ and $L^{\infty}(\Sigma)$,
$E^{\mathcal{A}}$ is idempotent and
$E^{\mathcal{A}}(L^{\infty}(\Sigma))=L^{\infty}(\mathcal{A})$ and
$E^{\mathcal{A}}(L^1(\Sigma))=L^1(\mathcal{A})$. Thus it can be
defined on all interpolation spaces of $L^1$ and $L^{\infty}$ such
as, Orlicz spaces \cite{besh}. If there is no possibility of
confusion, we write $E(f)$ in place of $E^{\mathcal{A}}(f)$. We list here some
of its useful properties:

\vspace*{0.2cm} \noindent $\bullet$ \  If $g$ is
$\mathcal{A}$-measurable, then $E(fg)=E(f)g$.

\noindent $\bullet$ \ $\varphi(E(f))\leq E(\varphi(f))$, where
$\varphi$ is a convex function.

\noindent $\bullet$ \ If $f\geq 0$, then $E(f)\geq 0$; if $f>0$,
then $E(f)>0$.

\noindent $\bullet$ \ For each $f\geq 0$, $S(f)\subseteq S(E(f))$.\\
For more information about conditional expectation one can see \cite{rao}.

For each $f\in L^{\Phi}(\Sigma)$ easily we get that 
$$\Phi(E(f))\leq
E(\Phi(f)),\ \ \ \ \ \ \ N_{\Phi}(E(f)\leq N_{\Phi}(f), \ \ \text{and} \ \ \ \|E(f)\|\leq \|f\|,$$
i.e, $E$ is a contraction on the Orlicz spaces.

\section{ \sc\bf Composition operators on Orlicz spaces}
Let $X$ be a Banach space and $\mathcal{B}(X)$ be the algebra of all  linear operators
on $X$. By an operator in $X$ we understand a
linear mapping $T:\mathcal{D}(T)\subseteq X\rightarrow
X$ defined on a linear subspace $\mathcal{D}(T)$ of
$X$ which is called the domain of $T$. The linear map $T$ is called densely defined if $\mathcal{D}(T)$ is dense in
$X$.\\
Let $(X,\Sigma,\mu)$ be a $\sigma$-finite measure space $f$ be a non-negative $\Sigma$-measurable function on $X$. Define the measure $\mu_f:\Sigma\rightarrow [0, \infty]$  by
$$\mu_f(E)=\int_{E}fd\mu, \ \ \ \ E\in \Sigma.$$
Throughout this paper, we denote by $(X,\Sigma,\mu)$, a measure space, that is, $X$ is a nonempty set, $\Sigma$ is a sigma algebra on $X$ and $\mu$ is a positive measure on $\Sigma$. Also, we assume that $\varphi:X\rightarrow X$ is a non-singular measurable transformation, that is, $\varphi^{-1}(F)\in \Sigma$, for every $F\in \Sigma$ and $\mu(\varphi^{-1}(F))=0$, if $\mu(F)=0$.
Non-singularity of $\varphi$ guarantees that the linear operator
$$C_{\varphi}:\mathcal{D}(C_{\varphi})\subseteq L^{\Phi}(\mu)\rightarrow L^{\Phi}(\mu), \ \  \ \ \ \ C_{\varphi}(f)=f\circ\varphi,$$
is well-defined on the Orlicz space $L^{\Phi}(\mu)$ and is called composition operator. For more details on composition operators on Orlicz spaces one can refer to \cite{chkm}.

Let $\mu\circ\varphi^{-1}$ be absolutely continuous with respect to $\mu$, $h_i=h_{\varphi^i}=\frac{d\mu\circ \varphi^{-i}}{d\mu}$ be the Radon-Nikodym derivative of $d\mu\circ \varphi^{-i}$ for $i\in\mathbb{N}$, with respect to $d\mu$ and $h=h_1$. So by definition we have
$$\mu\circ\varphi^{-i}(A)=\int_A h_{i} d\mu=\mu_{h_i}(A), \ \ \ \ \ \ \forall A\in \Sigma.$$
Also,
$$\mathcal{D}(C_{\varphi})=\{f\in L^p(\mu):\int_X|f\circ \varphi|^pd\mu<\infty\}$$
$$\|C_{\varphi}(f)\|^p_p=\int_X|f\circ \varphi|^pd\mu=\int_X|f|^phd\mu=\int_X(|f|h^{\frac{1}{p}})^pd\mu=\|M_{h^{\frac{1}{p}}}(f)\|^p_p.$$
This means that $\mathcal{D}(C_{\varphi})=\mathcal{D}(M_{h^{\frac{1}{p}}})$.
Since
$$\int_X|f|^p(1+h)d\mu=\int_X|f|^pd\mu+\int_X|f\circ\varphi|^pd\mu=\|f\|^p_p+\|C_{\varphi}(f)\|^p_p,$$
then we have $\mathcal{D}(C_{\varphi})=L^p((1+h)d\mu)=\mathcal{D}(M_{h^{\frac{1}{p}}})$. So if we set $\Phi(x)=\frac{x^p}{p}$, for each $p\geq1$ and $x\geq0$, then by Theorem 3.1 of \cite{eb}, we get that $\overline{L^p(\nu)}^{\|.\|_p}=\overline{\mathcal{D}(M_{h^{\frac{1}{p}}})}^{\|.\|_p}=L^p(\mu)$, where $d\nu=(1+h)d\mu$, provided that $h<\infty$, a.e., $\mu$. Moreover, if we set $J=1+h$ and suppose that $M_{h^{\frac{1}{p}}}:\mathcal{D}(M_{h^{\frac{1}{p}}})\subseteq L^{p}(\mu)\rightarrow L^p(\mu)$, then by Theorem 3.4 of \cite{eb} we get that the following statements are equivalent:
\begin{itemize}
  \item $M_{h^{\frac{1}{p}}}$ is densely defined on $L^p(\mu)$.
  \item $J-1=h<\infty$, a.e., $\mu$.
  \item $\mu_{J-1}$ is $\sigma$-finite.
\end{itemize}
So by the above observations we get the following proposition.
\begin{prop}
Let $1\leq p<\infty$, $d\nu=(1+h)\mu$ and $C_{\varphi}:\mathcal{D}(C_{\varphi})\subseteq L^p(\mu)\rightarrow L^p(\mu)$ be the composition operator. Then the following statements are equivalent:
\begin{itemize}
  \item $C_{\varphi}$ is densely defined on $L^p(\mu)$.
  \item $h<\infty$, a.e., $\mu$.
  \item $\mu_{h}$ is $\sigma$-finite.
\end{itemize}
Especially, $$\overline{L^p(\nu)}^{\|.\|_p}=\overline{\mathcal{D}(C_{\varphi})}^{\|.\|_p}=L^p(\mu)$$
provided that $h<\infty$, a.e., $\mu$.
\end{prop}
By the same way we can extend the above Proposition to the case $C_{\varphi}:\mathcal{D}(C_{\varphi})\subseteq L^p(\mu)\rightarrow L^q(\mu)$, where $1\leq p,q<\infty$. Indeed, for $1\leq p,q<\infty$, $C_{\varphi}:\mathcal{D}(C_{\varphi})\subseteq L^p(\mu)\rightarrow L^q(\mu)$ and $f\in \mathcal{D}(C_{\varphi})$ we have $\|C_{\varphi}(f)\|_q=\|M_{h^{\frac{1}{q}}}(f)\|_q$. And also by a straight forward calculations we get that
$M_{h^{\frac{1}{q}}}$ is densely defined from $L^p(\mu)$ into $L^q(\mu)$ if and only if  $h<\infty$, a.e., $\mu$, if and only if $\mu_{h}$ is $\sigma$-finite.

By the above observations the followings are equivalent:
\begin{itemize}
  \item $C_{\varphi}:\mathcal{D}(C_{\varphi})\subseteq L^p(\mu)\rightarrow L^q(\mu)$ is densely defined.
  \item $h<\infty$, a.e., $\mu$.
  \item $\mu_{h}$ is $\sigma$-finite.
\end{itemize}

Moreover, if $u:X\rightarrow \mathbb{C}$ is a measurable function, $\varphi:X\rightarrow X$ is a measurable transformation and the linear operator $uC_{\varphi}$ defined as follow is weighted composition operator.
$$uC_{\varphi}:\mathcal{D}(uC_{\varphi})\subseteq L^p(\mu)\rightarrow L^q(\mu), \ \ \ uC_{\varphi}(f)=u.f\circ\varphi, \ \ \ \forall f\in\mathcal{D}(uC_{\varphi})\subseteq L^p(\mu).$$
As we know for $f\in\mathcal{D}(uC_{\varphi})$ the followings hold:
\begin{align*}
\|uC_{\varphi}(f)\|^q_{q}&=\int_X|uC_{\varphi}(f)|^qd\mu\\
&=\int_XhE^{\varphi^{-1}(\Sigma)}(|u|^q)\circ\varphi^{-1}|f|^qd\mu\\
&=\int_XJ_h|f|^qd\mu\\
&=\|M_{J^{\frac{1}{q}}_{\varphi}}(f)\|^q_q,
\end{align*}
in which $J_{\varphi}=hE^{\varphi^{-1}(\Sigma)}(|u|^q)\circ\varphi^{-1}$.
Consequently we get that the following statements are equivalent:
\begin{itemize}
  \item $uC_{\varphi}:\mathcal{D}(uC_{\varphi})\subseteq L^p(\mu)\rightarrow L^q(\mu)$ is densely defined.
  \item $J_{\varphi}<\infty$, a.e., $\mu$.
  \item $\mu_{J_{\varphi}}$ is $\sigma$-finite.
\end{itemize}

Here we prove that the last two statements of the above equivalence are mutually equivalent that we will use it in Orlicz space setting.
\begin{lem}\label{l3.2}
Let $(X,\Sigma,\mu)$ be a $\sigma$-finite measure space and $\varphi:X\rightarrow X$ be a non-singular measurable transformation. Then $h<\infty$, a.e., $\mu$ if and only if the measure space $(X, \varphi^{-1}(\Sigma),\mu|_{\varphi^{-1}(\Sigma)})$ is a $\sigma$-finite measure space if and only if the measure space $(X, \varphi^{-1}(\Sigma),\mu_h)$ is a $\sigma$-finite measure space..
\end{lem}
\begin{proof}
Let $h<\infty$, a.e., $\mu$. Since $(X,\Sigma,\mu)$ is a $\sigma$-finite measure space, then we have $X=\cup^{\infty}_{n=1}A_n$, for $A_n\in \Sigma$, with $0<\mu(A_n)<\infty$. For $n,k\in \mathbb{N}$, we set $B_{n,k}=\{x\in A_n: h(x)\leq k\}$. Then we have
$$\mu\circ \varphi^{-1}(B_{n,k})=\int_{B_{n,k}}hd\mu\leq k\mu(B_{n,k})\leq k\mu(A_n)<\infty.$$
On the other hands
$$X=\cup^{\infty}_{n=1}\left(\cup^{\infty}_{k=1}\varphi^{-1}(B_{n,k})\right)\cup\varphi^{-1}(E_0),$$
in which $E_0=\{x\in X: h(x)=\infty\}$. This implies that $(X, \varphi^{-1}(\Sigma),\mu|_{\varphi^{-1}(\Sigma)})$ is a $\sigma$-finite measure space.\\
Conversely, assume that $(X, \varphi^{-1}(\Sigma),\mu|_{\varphi^{-1}(\Sigma)})$ is a $\sigma$-finite measure space. So $X=\cup^{\infty}_{n=1}\varphi^{-1}(B_n)$, where $B_n\in \Sigma$ and $\mu\circ\varphi^{-1}(B_n)<\infty$. Without lose of generality we can assume that the sequence $\{B_n\}$ is increasing. Let $E=\{x\in X: h(x)=\infty\}$, $E_n=\{x\in X: h(x)>n\}$ and $E_{n,k}=\{x\in B_n: h(x)>k\}$, for $n,k\in\mathbb{N}$. It is clear that the sequence $\{E_n\}_{n\in \mathbb{N}}$ is decreasing and $E=\cap^{\infty}_{n=1}E_n$. If $\mu(E)>0$, then there exists $k, n\in \mathbb{N}$ such that $\mu(E_{n,k})>0$. Hence we have
$$\mu\circ\varphi^{-1}(B_n)=\int_{B_n}hd\mu\geq\int_{E_{n,k}}hd\mu\geq k\mu(E_{n,k})=k\mu(B_n\cap E_k).$$
Hence $\mu\circ\varphi^{-1}(B_n)\geq\lim_{k\rightarrow \infty}k\mu(E)=\infty$ and this is a contradiction. This completes the proof.
\end{proof}

Now we have the next Theorem.
\begin{thm}\label{t3.3}
Let $\Phi$ be a \textit{Young' functions} and $\varphi:X\rightarrow X$ be a non-singular measurable transformation. Then the composition operator  $C_{\varphi}:\mathcal{D}(C_{\varphi})\subseteq L^{\Phi}(\mu)\rightarrow L^{\Phi}(\mu)$ is densely defined provided that $h<\infty$, a.e., $\mu$. Moreover, if $d\nu=(1+h)d\mu$, then $\overline{L^{\Phi}(\nu)}^{N_{\Phi}}=\overline{\mathcal{D}(C_{\varphi})}^{N_{\Phi}}=L^{\Phi}(\mu)$.
\end{thm}
\begin{proof}
Let $h<\infty$ a.e., $\mu$. Then $\mu(\{x\in X: h(x)=\infty\})=0$ and so
$$S(h)=\cup^{\infty}_{n=1}\left(\{x\in X: n-1\leq h(x)<n\}=A_n\right)=X.$$
It is obvious that $A_n$'s are disjoint. If $f\in L^{\Phi}(\mu)$, then by definition
$$\sum^{\infty}_{n=1}\int_{A_n}\Phi(kf)d\mu=\int_{X}\Phi(kf)d\mu<\infty,$$
for some $k>0$. Hence for each $\epsilon>0$, there exists $N>0$ such that
\begin{equation}\label{e1}
\sum^{\infty}_{n=N}\int_{A_n}\Phi(kf)d\mu<\epsilon.
\end{equation}
Let $B_N=\cup^{\infty}_{n=N}A_n$ and $C_N=\cup^{N-1}_{n=1}A_n$. Then

$$\int_{B_N}\Phi(kf)d\mu=\sum^{\infty}_{n=N}\int_{A_n}\Phi(kf)d\mu<\epsilon$$
and $C_N=\{x\in X: h(x)<N-1\}$. It is clear that $f_N=f.\chi_{C_N}\in L^{\Phi}(\mu)$.

 Now we show that $f_N\in \mathcal{D}(C_{\varphi}(f_N))$.
\begin{align*}
\int_X\Phi(\frac{C_{\varphi}(f_N)}{(N-1)N_{\Phi}(f)})d\mu&\leq\frac{1}{N-1}\int_{C_N}\Phi(\frac{f_N}{N_{\Phi}(f)})hd\mu\\
&\leq\int_{C_N}\Phi(\frac{f_N}{N_{\Phi}(f)})d\mu\\
&\leq\int_{X}\Phi(\frac{f}{N_{\Phi}(f)})d\mu\\
&\leq 1.
\end{align*}
Hence $N_{\Phi}(C_{\varphi}(f_N))\leq (N-1)N_{\Phi}(f)<\infty$ and so $f_N\in \mathcal{D}(C_{\varphi})$. It is clear that $C_N\nearrow X$, and so $f_N\rightarrow f$, almost every where and $|f|\leq |f_n|$, for each $n\in \mathbb{N}$. Hence by Lebesgue dominated convergence theorem, we get that $\rho_{\Phi}(f_N)\rightarrow \rho_{\Phi}(f)$, when $N\rightarrow \infty$. By these observations and Theorem \ref{t23} we get that $N_{\Phi}(f_N-f)\rightarrow 0$ as $N\rightarrow \infty$. Therefore $C_{\varphi}$ is densely defined on $L^{\Phi}(\mu)$.\\
Moreover, for $N>1$, $\int_X\Phi(\frac{f_N}{N.N_{\Phi}(f)})d\nu\leq 1$. Hence $f_N\in L^{\Phi}(\nu)$ and consequently
$$\overline{L^{\Phi}(\nu)}^{N_{\Phi}}=\overline{\mathcal{D}(C_{\varphi})}^{N_{\Phi}}=L^{\Phi}(\mu).$$

\end{proof}
Here we recall the following fundamental lemma.
\begin{lem}\label{lem32}
If $(X,\mathcal{A},\mu)$ is a $\sigma$-finite measure space and $f$ is an $\mathcal{A}$-measurable function such that $f<\infty$ a.e. $\mu$, then there exists a sequence $\{B_n\}^{\infty}_{n=1}\subseteq \mathcal{A}$ such that $\mu(B_n)<\infty$ and $f<n$ a.e. $\mu$ on $B_n$ for every $n\in \mathbb{N}$ and $B_n\nearrow X$ as $n\rightarrow \infty$.
\end{lem}

In the following Theorem we characterize densely defined Composition operators on Orlicz spaces.
\begin{thm} \label{t2}
Let $\Phi$ be a \textit{Young' functions}, $\varphi:X\rightarrow X$ be a non-singular measurable transformation and $C_{\varphi}:\mathcal{D}(C_{\varphi})\subseteq L^{\Phi}(\mu)\rightarrow L^{\Phi}(\mu)$ is the composition operator. If $d\nu=hd\mu$, then the following are equivalent:\\

(i) $C_{\varphi}$ is densely defined on $L^{\Phi}(\mu)$,\\

(ii) $h<\infty$ a.e., $\mu$.\\

(iii) $\mu|_{\varphi^{-1}(\Sigma)}$ is $\sigma$-finite.
\end{thm}
\begin{proof}
 $(i)\rightarrow (ii)$. By definition of $d\nu=(1+h)d\mu$, for each $A\in\Sigma$, we have $\nu(A)=\mu(A)+\int_Ahd\mu$. Since $h$ is a non-negative measurable function, then $\nu(A)=0$ if and only if $\mu(A)=0$. Let $E=\{x\in X: h(x) =\infty\}=\cap^{\infty}_{n=1}E_n$, in which $E_n=\{x\in X: h(x)>n\}$. Then for each $f\in L^{\Phi}(\nu)$ we must have $f\chi_{E}=0$ a.e., $\mu$, because if $f\chi_{E}.\chi_A\neq0$ a.e., $\mu$, for some $A\in \Sigma$ with $0<\mu(A)<\infty$, then by using the fact we can approximate $f$ with a function $g\in \mathcal{D}(C_{\varphi})$, we will have $N_{\Phi}(f)=\infty$. Hence $f.\chi_{E}.(1+h)=0$ a., $\mu$. As a result we have $(1+h).\chi_{A\cap E}=0$ a.e., $\mu$ for all measurable set $A\in \Sigma$ with $\mu(A)<\infty$. Since $(X, \Sigma, \mu)$ is $\sigma$-finite measure space, then we get that $(1+h).\chi_{E}=0$ a.e. $\mu$. Moreover, we have $S(1+h)=X$, this implies that $\chi_{E}=0$ a.e., $\mu$ and so $\mu(E)=0$.\\
The implication $(ii)\rightarrow (i)$ holds by Theorem \ref{t3.3}. Moreover, by the Lemma \ref{l3.2} we get that $(ii)$ and $(iii)$ are equivalent.
\end{proof}
We recall the fact that for every Young's function $\Phi$ and positive numbers $a,b\in \mathbb{R}$,
$$\Phi^{-1}(a+b)\leq \Phi^{-1}(a)+\Phi^{-1}(b), \ \ \ \ \ \ \ \Phi(a)+\Phi(b)\leq \Phi(a+b).$$
Also, by convexity of $\Phi$ we have
 $$\frac{1}{2}(\Phi^{-1}(2a)+\Phi^{-1}(2b))\leq \Phi^{-1}(a+b), \ \ \ \ \ \ \ \ \Phi(a+b)\leq \frac{1}{2}(\Phi(2a)+\Phi(2b)).$$
In the next Lemma we get that the reverse of these inequalities are also valid provided that $\Phi\in \Delta_2$.
\begin{lem}\label{l3.6}
Let $\Phi$ be a Young's function and $\Phi\in \Delta_2$. Then there exists positive numbers $K, L$ such that for all $a,b\in \mathbb{R}^+$ with $a\geq x_0, \  b\geq x_0$ ($x_0$ comes from the definition of $\Delta_2$ condition),
 $$\Phi^{-1}(a)+\Phi^{-1}(b)\leq L\Phi^{-1}(a+b), \ \ \ \ \ \ \ \ \Phi(a+b)\leq K(\Phi(a)+\Phi(b)).$$
 Moreover, if $\Phi\in \Delta_2$ globally, then these inequalities hold for all $a,b\in \mathbb{R}^+$.
\end{lem}
\begin{proof}
It is an easy exercise.
\end{proof}
Let $\varphi, \psi$ be non-singular measurable transformations on $X$ that are absolutely continuous with respect to $\mu$ and $h_1=\frac{d\mu\circ\varphi^{-1}}{d\mu}$, $h_2=\frac{d\mu\circ\psi^{-1}}{d\mu}$. Then by taking $J=1+h_1+h_2$ we have $L^{\Phi}(Jd\mu)=\mathcal{D}(\alpha_1C_{\varphi}+\alpha_2C_{\psi})$, for any $\alpha_1, \alpha_2\in \mathbb{R}^+$. Indeed, if $f\in L^{\Phi}(Jd\mu)$, then there exists $k>0$ such that
$$\int_X\Phi(kf)d\mu+\int_X\Phi(kf\circ\varphi)d\mu+\int_X\Phi(kf\circ\psi)d\mu=\int_X\Phi(kf)Jd\mu<\infty.$$
Hence
\begin{align*}
\int_X\Phi(\frac{k}{2(\alpha_1+\alpha_2)}(\alpha_1C_{\varphi}+\alpha_2C_{\psi})(f))d\mu
 &\leq \frac{1}{2}\int_X\Phi(\frac{\alpha_1k}{\alpha_1+\alpha_2}(f\circ\varphi))d\mu
 +\frac{1}{2}\int_X\Phi(\frac{k\alpha_2}{\alpha_1+\alpha_2}(f\circ\psi))d\mu\\
 &\leq\int_X\Phi(kf\circ\varphi)d\mu
 +\int_X\Phi(kf\circ\psi)d\mu <\infty.
\end{align*}
This means that $(\alpha_1C_{\varphi}+\alpha_2C_{\psi})(f)\in L^{\Phi}(\mu)$ and so $f\in \mathcal{D}(\alpha_1C_{\varphi}+\alpha_2C_{\psi})$. Thus we have $L^{\Phi}(Jd\mu)\subseteq\mathcal{D}(\alpha_1C_{\varphi}+\alpha_2C_{\psi})$. For the converse let $f\in \mathcal{D}(\alpha_1C_{\varphi}+\alpha_2C_{\psi})$. Then
$(\alpha_1C_{\varphi}+\alpha_2C_{\psi})(f)\in L^{\Phi}(\mu)$. So there exists $k>0$ such that
$$\int_X\Phi(k(\alpha_1C_{\varphi}+\alpha_2C_{\psi})(f))d\mu<\infty.$$
Now by taking $\beta=(k\alpha_1, k\alpha_2)$ we have
 \begin{align*}
 \int_X\Phi(\beta f)(J-1)d\mu&\leq \int_X\Phi(k\alpha_1C_{\varphi}f)d\mu+\int_X\Phi(k\alpha_2C_{\psi}f)d\mu\\
 &\leq\int_X\Phi(k(\alpha_1C_{\varphi}+\alpha_2C_{\psi})(f))d\mu<\infty.
 \end{align*}
 This implies that $f\in L^{\Phi}((J-1)d\mu)$ and since $f\in L^{\Phi}(\mu)$, then we have $f\in L^{\Phi}(Jd\mu)$. Hence we have  $L^{\Phi}(Jd\mu)\supseteq\mathcal{D}(\alpha_1C_{\varphi}+\alpha_2C_{\psi})$.\\
 By the above assumptions we also have $\mathcal{D}(C_{\varphi}\circ C_{\psi})=L^{\Phi}(J_0d\mu)$, in which $J_0=1+h_2+h_1\circ \psi^{-1}$.
\begin{lem}\label{l3.7}
Let $g$ be a non-negative measurable function on $X$ and $d\nu=gd\mu$. Then $L^{\Phi}(\nu)$ is dense in $L^{\Phi}(\mu)$ if and only if $g<\infty$, a.e., $\mu$.
\end{lem}
\begin{proof}
Let $\Phi\in \Delta_2$ and $g<\infty$, a.e., $\mu$. Then $\mu(\{x\in X: g(x)=\infty\})=0$ and $S(g)=\cup^{\infty}_{n=1}E_n$, where $E_n=\{x\in X: g(x)< n\}$. Let $f\in L^{\Phi}(\mu)$ and $f_n=f.\chi_{E_n}$. It is easy to see that $f_n\in L^{\Phi}(\nu)$, $|f_n|\leq |f|$ and $f_n\rightarrow f$ a.e., $\mu$. Then by monotone convergence Theorem and Theorem \ref{t23} we get that $N_{\Phi}(f_n-f)\rightarrow 0$ as $n\rightarrow \infty$. This implies that $L^{\Phi}(\nu)$ is dense in $L^{\Phi}(\mu)$. Conversely, suppose that $L^{\Phi}(\nu)$ is dense in $L^{\Phi}(\mu)$ and $F=\{x\in X: g(x)=\infty\}=\cap^{\infty}_{n=1}F_n$, where $E_n=\{x\in X: ga(x)>n\}$. Assume that $\mu(F)\geq 0$ and $f\in L^{\Phi}(\mu)$. Then for every $\epsilon>0$ there exists $h\in L^{\Phi}(\nu)$ such that $N_{\Phi}(f-h)<\epsilon$. Since $\mu(F)>0$, then $N_{\Phi}(h)=\infty$. This is a contradiction. Hence we must have $\mu(F)=0$ and this completes the proof.
\end{proof}

Here we recall that for Banach spaces $X,Y$ and densely defined linear operator $T:X\rightarrow Y$ there exists a unique maximal operator $T^{\ast}:\mathcal{D}(T^{\ast})\subset Y^{\ast}\rightarrow X^{\ast}$ such that
$$y^{\ast}(Tx)=\langle Tx, y^{\ast}\rangle=\langle x, T^{\ast}y^{\ast}\rangle=T^{\ast}y^{\ast}(x), \ \ \  x\in \mathcal{D}(T), \ \ \ y^{\ast}\in\mathcal{D}(T^{\ast}).$$
The linear operator $T^{\ast}$ is called the adjoint of $T$.\\
A linear operator $T:X\rightarrow Y$ is said to be closed if the graph of $T$, $\mathcal{G}(T)$ is closed in $X\times Y$, in which $\mathcal{G}(T)=\{(x,Tx):x\in \mathcal{D}(T)\}$. Also, $T$ is called closable, if there is a closed operator $\overline{T}$ with $\mathcal{G}(\overline{T})=\overline{\mathcal{G}(T)}$. Indeed, $\overline{T}$ is the smallest closed extension of $T$.
Here we recall the next theorem for later use.
\begin{thm}\cite{kato}\label{tk}
   Let $X,Y$ be reflexive Banach spaces. If $T:X\rightarrow Y$ is densely defined and closable, then $T^{\ast}$ is closed, densely defined and $T^{\ast\ast}=\overline{T}$.
   \end{thm}
In the next proposition we get that every densely defined composition operator $C_{\varphi}:\mathcal{D}(C_{\varphi})\subseteq L^{\Phi}(\mu)\rightarrow L^{\Phi}(\mu)$ is closed.
\begin{prop}\label{p3.9}
Let $C_{\varphi}$ be densely defined on the Orlicz space $L^{\Phi}(\mu)$. Then $C_{\varphi}$ is a closed operator.
\end{prop}
\begin{proof}
Let $\{f_n\}_{n\in \mathbb{N}}\subseteq \mathcal{D}(C_{\varphi})$ and $f,g\in L^{\Phi}(\mu)$ such that
$$f_n\rightarrow f\ \ \ \text{and}\ \ \ C_{\varphi}(f_n)\rightarrow g, \ \ \ \text{as} \ n\rightarrow \infty.$$
Since $N_{\Phi}(f_n-f)\rightarrow 0$, then $\rho_{\Phi}(f_n)\rightarrow \rho_{\Phi}(f)$. So there exists a subsequence $\{f_{n_k}\}_{k\in \mathbb{N}}$ such that $\Phi(f_{n_k})\rightarrow \Phi(f)$, a.e., $\mu$. Then we get that $f_{n_k}\rightarrow f$, a.e., $\mu$ and so $f_{n_k}\circ\varphi\rightarrow f\circ\varphi$, a.e., $\mu$. On the other hand we have $f_n\circ\varphi\rightarrow g$, so (without lose of generality) we get that $f_{n_k}\circ\varphi\rightarrow f\circ\varphi$, a.e., $\mu$. This implies that $g=f\circ\varphi$.
\end{proof}
Let $\Phi\in \Delta_2$ and $C_{\varphi}$ be densely defined on the Orlicz spaces $L^{\Phi}(\mu)$. Then $(L^{\Phi}(\mu))^*=L^{\Psi}(\mu)$, where $\Psi$ is the complementary Young's function to $\Phi$, by Riesz representation theorem we get that $C^*_{\varphi}=M_{h}C_{\varphi^{-1}}E^{\varphi}$, where $E^{\varphi}=E^{\varphi^{-1}(\Sigma)}$. In the next theorem we obtain the $\mathcal{D}(C^*_{\varphi})$ as the adjoint operator of $C_{\varphi}:L^{\Phi}(\mu)\rightarrow L^{\Phi}(\mu)$.
\begin{thm}
Let $\Phi\in \Delta_2$ and $C_{\varphi}:\mathcal{D}(C_{\varphi})\subseteq L^{\Phi}(\mu)\rightarrow L^{\Phi}(\mu)$ be densely defined. Then the adjoint operator $C^*_{\varphi}$ of $C_{\varphi}$ is as follow:
$$C^*_{\varphi}:\mathcal{D}\subseteq L^{\Psi}(\mu)\rightarrow L^{\Psi}(\mu), \ \ \ \ C^*_{\varphi}(f)=M_hC_{\varphi^{-1}}E^{\varphi}(f), \ \ \ f\in \mathcal{D}(C^*_{\varphi}).$$
Moreover, if $\Psi\in \Delta'$ and $\varphi$ is bijective, then
$$\mathcal{D}(C^*_{\varphi})\supseteq L^{\Psi}(X,\Sigma,\nu)\ \ \text{where} \ \ \ d\nu=(1+E^{\varphi}(h_{-1})\Psi(h)\circ\varphi)d\mu\ \text{and} \  h_{-1}=\frac{d\mu\circ\varphi}{d\mu}.$$
 And also $C^*_{\varphi}$ is densely defined if and only if $J=1+E^{\varphi}(h_{-1})\Psi(h)\circ\varphi<\infty$, a.e., $\mu$.
 \end{thm}
\begin{proof}
Since $\Phi\in \Delta_2$, then by Riesz representation theorem we have $(L^{\Phi}(\mu))^*=L^{\Psi}(\mu)$, where $\Psi$ is the complementary Young's function to $\Phi$. More precisely, for each $L\in (L^{\Phi}(\mu))^*$, there exists $g\in L^{\Psi}(\mu)$ such that $L(f)=\int_X f.gd\mu=\langle f,g\rangle$, for all $f\in L^{\Phi}(\mu)$. Hence for each $f\in \mathcal{D}(C_{\varphi})\subseteq L^{\Phi}(\mu)$ and $f\in \mathcal{D}((C_{\varphi})^*)\subseteq L^{\Psi}(\mu)$, we have
\begin{align*}
\langle f,C^*_{\varphi}(g)\rangle&=\langle C_{\varphi}f, g\rangle\\
&=\int_Xf\circ\varphi.gd\mu\\
&=\int_Xf.hE^{\varphi}(g)\circ\varphi^{-1}d\mu\\
&=\langle f, hE^{\varphi}(g)\circ\varphi^{-1}\rangle.
\end{align*}
Since $\mathcal{D}(C_{\varphi})$ is dense in $L^{\Phi}(\mu)$, the we have $C^*_{\varphi}(g)=hE^{\varphi}(g)\circ\varphi^{-1}=M_hC_{\varphi^{-1}}E^{\varphi}(g)$.
Let $g\in \mathcal{D}(C^*_{\varphi})$. Then by the assumption $\Psi\in \Delta'$
\begin{align*}
\int_X\Psi(hE^{\varphi}(g)\circ\varphi^{-1})d\mu&=\int_X\Psi(h\circ\varphi E^{\varphi}(g))\circ\varphi^{-1}d\mu\\
&=\int_X\Psi(h\circ\varphi E^{\varphi}(g))h_{-1}d\mu\\
&\leq\int_X\Psi(h\circ\varphi)\Psi(E^{\varphi}(g))h_{-1}d\mu\\
&=\int_X\Psi(E^{\varphi}(g))h_{-1}\Psi(h\circ\varphi)d\mu\\
&\leq\int_X\Psi(g)E^{\varphi}(h_{-1})\Psi(h\circ\varphi)d\mu.
\end{align*}
By the above observations we have $L^{\Psi}(X, \Sigma,\nu)\subseteq\mathcal{D}(C^{*}_{\varphi})\subseteq L^{\Psi}(\mu)$. Now the Lemma \ref{l3.7} implies that $C^*_{\varphi}$ is densely defined if and only if $J=1+E^{\varphi}(h_{-1})\Psi(h)\circ\varphi<\infty$, a.e., $\mu$.
\end{proof}
Finally in the next Theorem we characterize densely defined continuous composition operators.
\begin{thm}\label{t3.11}
If we consider the composition operator $C_{\varphi}$ on the Orlicz space $L^{\Phi}(\mu)$ such that $h<\infty$ a.e., $\mu$, and $\Psi\in\Delta'$ then   $C_{\varphi}:\mathcal{D}(C_{\varphi})\rightarrow L^{\Phi}(\mu)$ is continuous if and only if it is every where defined i.e., $\mathcal{D}(C_{\varphi})=L^{\Phi}(\mu)$.
\end{thm}
\begin{proof} Since  $h<\infty$ a.e., $\mu$, then the composition operator $C_{\varphi}$ is densely defined and by the Proposition \ref{p3.9} it is closed. Now, if $C_{\varphi}$ is continuous, then we get that $\mathcal{D}(C_{\varphi})$ is closed and so $\mathcal{D}(C_{\varphi})=L^{\Phi}(\Sigma)$.

Conversely, let $C_{\varphi}$ be every where defined. Since the composition operator $C_{\varphi}$ is closed, then by closed graph theorem we obtain that $C_{\varphi}$ is continuous.
\end{proof}

\smallskip\noindent
\textbf{Acknowledgement.} My manuscript has no associate data.

\end{document}